\documentclass[11pt]{amsart}
\usepackage[utf8]{inputenc}
\usepackage[english]{babel}
\usepackage[margin=1in]{geometry}
\usepackage{amsrefs}
\usepackage{tikz-cd}
\usepackage[hidelinks]{hyperref}
\usepackage{fancyhdr}
\usepackage{booktabs}
\theoremstyle{plain}
\newtheorem{theorem}{Theorem}[section]

\newtheorem{prop}[theorem]{Proposition}
\newtheorem{cor}[theorem]{Corollary}
\newtheorem{conj}[theorem]{Conjecture}
\theoremstyle{definition}
\newtheorem{defn}{Definition}[section]
\title{Lower Bounds on Intersection Families for Certain Graphs}
\author[P. Hamrick]{Paul Hamrick}
\address{Department of Mathematics, University of North Carolina at Chapel Hill, Chapel Hill, NC 27599}
\email{hamri@unc.edu}

\author[G. Hu]{Gary Hu}
\address{Department of Mathematics, Williams College, Williamstown, MA 01267}
\email{gh7@williams.edu}

\begin{document}

\begin{abstract}
A family of graphs $\mathcal{F}$ is $H$-intersecting if the edge intersection of any two graphs in $\mathcal{F}$ contains a copy of a fixed graph $H$. A fundamental problem is to determine the maximum size of such a family. The trivial lower bound of $2^{\binom{n}{2} - e(H)}$ is known to be not sharp for some graphs, such as the $P_4$ graph, as shown by Christofides. This paper presents two main contributions. First, we introduce a general construction for $H$-intersecting families based on decompositions of complete multipartite graphs, yielding new lower bounds for $H = K_{s_1, \dots, s_{k-1}, t}$. We compare this construction to a result by Balogh and Linz, showing that our bound is valid for a substantially wider range of parameters (beginning at $t \ge 2^{\sum_i s_i}$) and provides a stronger numerical bound for a large interval where both constructions are applicable. Second, we conjecture the $\frac{17}{128}$ Christofides bound for $P_4$ is optimal, which would resolve the Alon-Spencer conjecture. We computationally verify this density is optimal for families generated by connected 6-vertex host graphs with 7 or 8 edges.
\end{abstract}

\maketitle

\section{Introduction}

\begin{defn}
    A family of graphs $\mathcal{F}$, whose members are subgraphs of the complete graph $K_n$, is said to be \textbf{$H$-intersecting} if for any two graphs $F, F' \in \mathcal{F}$, their edge intersection $F \cap F'$ contains a subgraph isomorphic to $H$. 
\end{defn}

One generalization of a classic result in extremal combinatorics, the Erd\H{o}s-Ko-Rado Theorem, is to determine $\max |\mathcal{F}|$, the maximum possible size of an $H$-intersecting family, for different choices of $H$. A trivial construction for an $H$-intersecting family involves fixing a specific copy of $H$ within $K_n$ and taking all $2^{\binom{n}{2} - e(H)}$ supergraphs. This establishes a baseline lower bound of $2^{\binom{n}{2} - e(H)}$. 

For some graphs, this bound is sharp. A notable example is $H=K_3$, for which a conjecture by Simonovits and Sós stating this optimality was proven by Ellis, Filmus, and Friedgut \cite{ellis-filmus-friedgut}. However, this trivial bound is not always optimal. For the path on four vertices, $P_4$, the trivial bound is $2^{\binom{n}{2} - 3}$. This was also conjectured to be optimal by Simonovits and Sós \cite{chung-graham}.

\begin{conj}[Simonovits-Sós, Disproven]
If $\mathcal{F}$ is a $P_4$-intersecting family of labeled subgraphs of $K_n$, then $|\mathcal{F}| \le 2^{\binom n2 - 3}$.
\end{conj}
In 2012, Christofides \cite{christofides} disproved this conjecture, using the graph shown in Figure \ref{fig:christofides}.

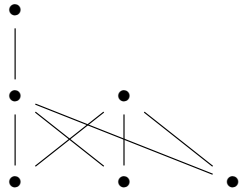
\begin{figure}[h]
    \centering
    \begin{tikzcd}
        \bullet \\
        \bullet & \bullet \\
        \bullet & \bullet & \bullet
        \arrow[no head, from=1-1, to=2-1]
        \arrow[no head, from=2-1, to=3-1]
        \arrow[no head, from=3-1, to=2-2]
        \arrow[no head, from=3-2, to=2-1]
        \arrow[no head, from=3-2, to=2-2]
        \arrow[no head, from=3-3, to=2-1]
        \arrow[no head, from=3-3, to=2-2]
    \end{tikzcd}
    \caption{The 7-edge host graph on 6 vertices used in the Christofides construction \cite{christofides}.}
    \label{fig:christofides}
\end{figure}

Christofides identified a collection $\mathcal{G}$ of 17 subgraphs of this host graph that were pairwise $P_4$-intersecting. By taking all supergraphs in $K_n$ ($n \geq 6$) for each graph in $\mathcal{G}$, he constructed a family of size $17 \cdot 2^{\binom{n}{2} - 7} = \frac{17}{128} \cdot 2^{\binom{n}{2}}$. Since $\frac{17}{128} > \frac{1}{8} = 2^{-3}$, this disproved the conjecture.

\begin{theorem}[Christofides]
For $n\ge 6$, there exists a $P_4$-intersecting family $\mathcal{F}$ with $|\mathcal{F}| \ge \frac{17}{128} \cdot 2^{\binom n2}$.
\end{theorem}

This method of generating a large intersecting family on $K_n$ from a dense intersecting subfamily of a small host graph is a useful technique. If one can find a graph $G$ with $m$ edges and an $H$-intersecting family $\mathcal{G}$ of its subgraphs with density $c = |\mathcal{G}|/2^m > 2^{-e(H)}$, then for sufficiently large $n$, there is an $H$-intersecting family on $K_n$ with size at least $c \cdot 2^{\binom n2}$.

\section{A Construction via Multipartite Graphs}

We present a construction that improves upon the trivial bound for intersecting families where $H$ is a complete multipartite graph. An alternative construction for non-trivial $H$-intersecting families, where $H = K_{s_1, \dots, s_{k-1}, t}$, was recently proved in \cite{balogh-linz} by Balogh and Linz. We do a detailed comparison between the two results in Section \ref{sec:Linz-Balogh}.

Here is our result:

\begin{theorem} \label{thm:main_construction}
Let $k \geq 2$ be an integer, and let $s_1, \dots, s_{k-1}$ be positive integers. Consider an integer $t \geq 2^{s_1 + \cdots + s_{k-1}}$, and let $H = K_{s_1, \dots, s_{k-1}, t}$ be a complete $k$-partite graph with parts of sizes $s_1, \dots, s_{k-1}$, and $t$. Denote by $N = |E(H)|$ the number of edges in $H$. Then, for any integer $n \geq s_1 + \cdots + s_{k-1} + t + 2$, there exists an $H$-intersecting family of subgraphs with size at least
\[
\frac{(t + 2)(2^{s_1 + \cdots + s_{k-1}} - 1) + 1}{2^{|E(K_{s_1, \dots, s_{k-1}, t+2})|}} 2^{\binom{n}{2}} > \frac{1}{2^N} 2^{\binom{n}{2}}.
\]
\end{theorem}

\begin{proof}
Let the host graph be $G = K_{s_1, \dots, s_{k-1}, t+2}$. The vertex set of $G$ is partitioned into sets $V_1, \dots, V_{k-1}, W$ with $|V_i|=s_i$ and $|W|=t+2$. The graph $H$ is an induced subgraph of $G$. Consider the $t+2$ subgraphs of $G$ of the form $G_w = G[V(G)\setminus\{w\}]$ for each $w \in W$. Each $G_w$ is isomorphic to $K_{s_1, \dots, s_{k-1}, t+1}$. The intersection of any two distinct subgraphs, $G_w \cap G_{w'}$, is isomorphic to $K_{s_1, \dots, s_{k-1}, t}$, which is precisely $H$.

We construct an $H$-intersecting family $\mathcal{G}$ of subgraphs of $G$ as follows. Let $m=s_1+\dots + s_{k-1}$. For each of the $t+2$ subgraphs $G_w$, we take all supergraphs of $G_w$ within $G$ except for $G$ itself. This gives $2^{e(G)-e(G_w)}-1 = 2^m-1$ distinct subgraphs for each $w$. Adding the complete graph $G$ to this collection gives a family of size $|\mathcal{G}| = (t+2)(2^m - 1) + 1$. Since the intersection of any two chosen subgraphs must contain some $G_w \cap G_{w'}$ or $G_w$ itself, this family $\mathcal{G}$ is $H$-intersecting.

To obtain the family $\mathcal{F}$ on $K_n$, we take all supergraphs of the members of $\mathcal{G}$. The measure of this family is $|\mathcal{G}| \cdot 2^{-e(G)}$. We must show this improves the trivial bound:
\[ |\mathcal{G}| \cdot 2^{-e(G)} > 2^{-e(H)} \iff |\mathcal{G}| > 2^{e(G)-e(H)} \]
The difference in edges is $e(G) - e(H) = m(t+2) - mt = 2m$. The inequality to verify is $(t+2)(2^m-1)+1 > 2^{2m}$.
Given the condition $t \ge 2^m$, we have:
\[ (t+2)(2^m-1)+1 \ge (2^m+2)(2^m-1)+1 = (2^{2m} + 2^m - 2) + 1 = 2^{2m} + 2^m - 1 \]
Since $m \ge 1$, this is strictly greater than $2^{2m}$, which completes the proof.
\end{proof}

In particular, if $k = 2$, this gives an improvement on the trivial bound for certain bipartite graphs.

\begin{cor}
If $s\in \mathbb N_{>0}$, $t\ge 2^s$, and $n\ge s+t+2$, there exists a $K_{s,t}$-intersecting family of size at least 
\[\frac{(t+2)(2^s-1)+1}{2^{s(t+2)}}2^{\binom n2} > \frac{1}{2^{st}}2^{\binom n2}.\]
\end{cor}

The proof technique can be generalized as follows.

\begin{prop}
Let $H$ be a graph. Suppose there exists a host graph $G$ and a set of $N$ distinct subgraphs $\{H_1, \dots, H_N\}$ of $G$ that satisfy the following conditions:
\begin{enumerate}
    \item Intersection Property: For any $i, j \in \{1, \dots, N\}$, the intersection $H_i \cap H_j$ contains a subgraph isomorphic to $H$.
    \item Disjoint Complement Property: For any $i \neq j$, the union $E(H_i) \cup E(H_j) = E(G)$.
\end{enumerate}
Then for any integer $n \ge |V(G)|$, there exists an $H$-intersecting family $\mathcal{F}$ of subgraphs of $K_n$ of size
\[ \left(1 + \sum_{i=1}^N \left(2^{|E(G)|-|E(H_i)|}-1\right)\right) \cdot 2^{\binom{n}{2}-|E(G)|}. \]
\end{prop}

\begin{proof}
    Let $\mathcal{G}$ be a family of subgraphs of $G$ defined as follows:
\[ \mathcal{G} = \{G\} \cup \bigcup_{i=1}^{N} \{ F \mid E(H_i) \subseteq E(F) \text{ and } F \neq G \}. \]
This family consists of the graph $G$ itself, along with every proper supergraph of each $H_i$.

Now, we show that the family $\mathcal{G}$ is $H$-intersecting. Let $F_1, F_2$ be two distinct members of $\mathcal{G}$. We must show that $F_1 \cap F_2$ contains a subgraph isomorphic to $H$. Any element of $\mathcal G$ contains some $H_i$ as a subgraph, so let $F_1$ have $H_{i_1}$ as a subgraph and $F_2$ have $H_{i_2}$ as a subgraph. Then $F_1\cap F_2$ contains $H_{i_1}\cap H_{i_2}$, which we know contains a subgraph isomorphic to $H$. Thus, $\mathcal{G}$ is an $H$-intersecting family.

Next, we show that the size of $\mathcal{G}$ is $1 + \sum_{i=1}^{N} \left(2^{|E(G)|-|E(H_i)|}-1\right)$. Let $\mathcal{S}_i = \{ F \mid E(H_i) \subseteq E(F) \text{ and } F \neq G \}$. Then $\mathcal{G} = \{G\} \cup \bigcup_{i=1}^N \mathcal{S}_i$. To calculate $|\mathcal{G}|$, we show that the sets $\mathcal{S}_i$ are pairwise disjoint for $i \neq j$. Assume for contradiction that there exists a graph $F \in \mathcal{S}_i \cap \mathcal{S}_j$ for $i \neq j$. By definition of $\mathcal{S}_i$ and $\mathcal{S}_j$, it must be that $E(H_i) \subseteq E(F)$ and $E(H_j) \subseteq E(F)$. This implies $E(H_i) \cup E(H_j) \subseteq E(F)$. From Property 2, we know $E(H_i) \cup E(H_j) = E(G)$. Therefore, $E(G) \subseteq E(F)$. Since $F$ is a subgraph of $G$, we must have $E(F) = E(G)$, which implies $F=G$. This contradicts the fact that the sets $\mathcal{S}_i$ and $\mathcal{S}_j$ explicitly exclude $G$. The sets are thus pairwise disjoint. The size of $\mathcal{G}$ is \[ |\mathcal{G}| = |\{G\}| + \sum_{i=1}^{N} |\mathcal{S}_i| = 1 + \sum_{i=1}^{N} \left(2^{|E(G)| - |E(H_i)|} - 1\right). \] This confirms the size calculation.

Finally, we construct the family $\mathcal{F}$ on the $n$ vertices of $K_n$. We embed $G$ in $K_n$. The family $\mathcal{F}$ is the set of all supergraphs in $K_n$ of the members of $\mathcal{G}$. Each graph $F_G \in \mathcal{G}$ can be extended to a supergraph in $K_n$ by adding any subset of the $\binom{n}{2} - |E(G)|$ edges that exist in $K_n$ but not in $G$. Since all members of $\mathcal{G}$ are distinct, their corresponding sets of supergraphs in $K_n$ are also disjoint.

The total size of $\mathcal{F}$ is therefore $|\mathcal{G}| \cdot 2^{\binom{n}{2} - |E(G)|}$. Substituting the expression for $|\mathcal{G}|$ yields the stated result.

\end{proof}

\section{Comparison to the Balogh-Linz Bound}\label{sec:Linz-Balogh}

In this section, we compare our construction (from Theorem \ref{thm:main_construction}) with an alternative method for generating non-trivial $H$-intersecting families for complete multipartite graphs. This alternative was recently provided by Balogh and Linz \cite{balogh-linz}. 

Let $H = K_{s_1, \dots, s_r, t}$ be the target graph and let $S = \sum_{i=1}^{r} s_i$. The trivial bound for an $H$-intersecting family is $2^{\binom{n}{2} - |E(H)|}$. The Balogh-Linz result is stated as follows.

\begin{theorem}[Balogh-Linz Bound, \cite{balogh-linz}] \label{thm:balogh-linz}
Let $s_1,s_2,\dots,s_r,t$ be integers with $s_i \ge 1$ for $1\le i \le r$ and $t>2^{2\sum_i s_i} - 2\sum_i s_i - 1$. Then there exists a $K_{s_1,\dots,s_r,t}$-intersecting family of graphs $\mathcal H$ with
\[|\mathcal H| > 2^{\binom n2 - \sum_{1\le i < j \le r}s_is_j - \sum_{i=1}^r s_it}.\]
\end{theorem}

The construction in \cite{balogh-linz} establishes a family whose size is $(t+2S+1)^S \cdot 2^{\binom{n}{2} - |E(H)| - 2S^2}$. The improvement factor over the trivial bound, which we denote $f_{BL}(t)$, is therefore:
\[
    f_{BL}(t) = \frac{(t+2S+1)^S}{2^{2S^2}}.
\]
Note that $f_{BL}(t)$ is a polynomial in $t$ of degree $S$.

In contrast, our construction from Theorem \ref{thm:main_construction} provides a family of size $((t+2)(2^S-1)+1) \cdot 2^{\binom{n}{2} - |E(H)| - 2S}$. This yields a linear improvement factor, which we denote $f_{\text{new}}(t)$:
\[
    f_{\text{new}}(t) = \frac{(t+2)(2^S-1)+1}{2^{2S}}.
\]
This factor $f_{\text{new}}(t)$ is linear in $t$.

This distinction in the degree of $t$ implies that the Balogh-Linz construction is asymptotically stronger for sufficiently large $t$. Simultaneously, an analysis of the respective conditions on $t$ reveals that our construction is valid across a wider range of parameters.

For example, the Balogh-Linz construction requires $t > 2^{2S} - 2S - 1$. Our construction requires only $t \ge 2^S$. This creates a substantial interval $t \in [2^S, 2^{2S} - 2S - 1]$ where our construction is the only one applicable.

For $t \ge 2^{2S} - 2S$, both constructions are valid. We can compute the crossover point, $t_{\text{cross}}$, where $f_{BL}(t)$ becomes larger than $f_{\text{new}}(t)$.

Table \ref{tab:comparison} provides a comprehensive numerical analysis for small values of $S$. All crossover points are computed by solving $f_{BL}(t) > f_{\text{new}}(t)$ for the smallest integer $t$.

\begin{table}[h]
\centering
\caption{Comparison of applicability and performance of $f_{\text{new}}(t)$ and $f_{BL}(t)$.}
\label{tab:comparison}
\begin{tabular}{@{}lcccc@{}}
\toprule
 & $f_{\text{new}}(t)$ Validity & $f_{BL}(t)$ Validity & Validity Gap & Crossover $t_{\text{cross}}$ \\
$S$ & (Our Bound) & (Balogh-Linz Bound) & (Our Bound Only) & (where $f_{BL} > f_{\text{new}}$) \\
\midrule
2 & $t \ge 4$ & $t \ge 12$ & $[4, 11]$ & $t = 41$ \\
3 & $t \ge 8$ & $t \ge 58$ & $[8, 57]$ & $t = 160$ \\
4 & $t \ge 16$ & $t \ge 248$ & $[16, 247]$ & $t = 621$ \\
5 & $t \ge 32$ & $t \ge 1014$ & $[32, 1013]$ & $t = 2404$ \\
\bottomrule
\end{tabular}
\end{table}

The data in Table \ref{tab:comparison} leads to two conclusions:
\begin{itemize}
    \item The validity gap, where our construction is the only one known to improve the trivial bound, grows doubly exponentially with $S$. For $S=5$, our bound applies to a range of 982 values of $t$ for which the Balogh-Linz construction is not valid.
    
    \item Even after the Balogh-Linz construction becomes valid, our linear bound remains numerically superior for a significant subsequent interval. For example, at the first point of Balogh-Linz validity for $S=5$ (i.e., $t=1014$), our improvement factor is $f_{\text{new}}(1014) \approx 30.76$, whereas the Balogh-Linz factor is $f_{BL}(1014) = (1025/1024)^5 \approx 1.005$.
\end{itemize}

We can estimate the upper point of this range by equating the two improvement factors:
\[\frac{t}{2^S} \approx f_{\text{new}}(t) \approx f_{BL}(t) \approx \frac{t^S}{2^{2S^2}}.\]
Rearranging gives $t^{S-1} \approx 2^{2S^2 - S}$, or $t\approx 2^{2S + 1 + \frac{1}{S-1}}$. This approximation aligns closely with the actual crossover points; for $S = 5$, it yields $t\approx 2^{11 + \frac 14} \approx 2435$, which is close to the actual value of 2404. Thus, our construction typically improves on the Balogh-Linz construction for $t$ up to approximately $2^{2S+1}$.

In summary, while the Balogh-Linz construction is asymptotically more powerful, Theorem \ref{thm:main_construction} provides a bound that is effective for a wider range of parameters. It exclusively improves the trivial bound in the interval $t \in [2^S, 2^{2S} - 2S - 1]$ and remains numerically stronger for a large interval thereafter.

\section{Applications and New Bounds}

Theorem \ref{thm:main_construction} provides a rich source of new bounds for intersecting families where the target graph is not a star-forest. We select the smallest integer $t$ satisfying the theorem's condition, $t = 2^s$, to maximize the improvement over the trivial bound.

\subsection{Complete Bipartite Graphs \texorpdfstring{($H = K_{s,t}$)}{($H = Kst$)}}

The case $k=2$ yields bounds for complete bipartite graphs. Here $s_1 = s$, so we set $S = s$. We select the smallest integer $t$ satisfying the theorem's condition, $t = 2^S = 2^s$.

\begin{cor}
The following improved lower bounds for $K_{s, 2^s}$-intersecting families hold:
\begin{enumerate}
    \item For $H = K_{2,4}$ ($n \ge 8$): $|\mathcal{F}| \ge \frac{19}{2^{12}} \cdot 2^{\binom{n}{2}}$ (vs. trivial $\frac{16}{2^{12}}$).
    \item For $H = K_{3,8}$ ($n \ge 13$): $|\mathcal{F}| \ge \frac{71}{2^{30}} \cdot 2^{\binom{n}{2}}$ (vs. trivial $\frac{64}{2^{30}}$).
    \item For $H = K_{4,16}$ ($n \ge 22$): $|\mathcal{F}| \ge \frac{271}{2^{72}} \cdot 2^{\binom{n}{2}}$ (vs. trivial $\frac{256}{2^{72}}$).
    \item For $H = K_{5,32}$ ($n \ge 39$): $|\mathcal{F}| \ge \frac{1055}{2^{170}} \cdot 2^{\binom{n}{2}}$ (vs. trivial $\frac{1024}{2^{170}}$).
\end{enumerate}
\end{cor}
\begin{proof}
We prove (4) as an example. Let $s=5$, so $t=2^5=32$. The target is $H=K_{5,32}$ with $|E(H)|=160$. The host graph is $G=K_{5,34}$ with $|E(G)|=170$. The density factor from our construction is
\[ \frac{(t+2)(2^s-1)+1}{2^{|E(G)|}} = \frac{(32+2)(2^5-1)+1}{2^{170}} = \frac{1055}{2^{170}}. \]
The trivial bound gives a density of $\frac{1}{2^{160}} = \frac{1024}{2^{170}}$. Since $1055 > 1024$, our bound is better. The other proofs are analogous.
\end{proof}

\subsection{General Multipartite Graphs}

The construction is equally effective for multipartite graphs with more than two parts.

\begin{cor}
The following improved lower bounds for multipartite-intersecting families hold:
\begin{enumerate}
    \item For $H = K_{1,1,4}$ ($n \ge 8$): $|\mathcal{F}| \ge \frac{19}{2^{13}} \cdot 2^{\binom{n}{2}}$ (vs. trivial $\frac{16}{2^{13}}$).
    \item For $H = K_{1,2,8}$ ($n \ge 13$): $|\mathcal{F}| \ge \frac{71}{2^{32}} \cdot 2^{\binom{n}{2}}$ (vs. trivial $\frac{64}{2^{32}}$).
    \item For $H = K_{2,2,16}$ ($n \ge 22$): $|\mathcal{F}| \ge \frac{271}{2^{76}} \cdot 2^{\binom{n}{2}}$ (vs. trivial $\frac{256}{2^{76}}$).
    \item For $H = K_{1,1,1,8}$ ($n \ge 13$): $|\mathcal{F}| \ge \frac{71}{2^{33}} \cdot 2^{\binom{n}{2}}$ (vs. trivial $\frac{64}{2^{33}}$).
\end{enumerate}
\end{cor}
\begin{proof}
We prove (3) as an example. Let $k=3$ with $s_1=2, s_2=2$. Then $S=4$ and we set $t=2^4=16$. The target is $H=K_{2,2,16}$ with $|E(H)| = 2\cdot2 + 2\cdot16 + 2\cdot16 = 68$. The host is $G=K_{2,2,18}$ with $|E(G)| = 2\cdot2 + 2\cdot18 + 2\cdot18 = 76$. The new density is
\[ \frac{(t+2)(2^S-1)+1}{2^{|E(G)|}} = \frac{(16+2)(2^4-1)+1}{2^{76}} = \frac{18 \cdot 15 + 1}{2^{76}} = \frac{271}{2^{76}}. \]
The trivial bound gives a density of $\frac{1}{2^{68}} = \frac{256}{2^{76}}$. 
\end{proof}

\section{On the Optimality of the \texorpdfstring{$P_4$}{P4} Bound}

While the Christofides construction improved the known bound for $P_4$-intersecting families from the trivial $\frac{1}{8}$ to $\frac{17}{128}$, its optimality remains an open question. It is believed among some experts in the field that this bound is indeed sharp, though this conjecture does not appear to be formally stated in the literature. We state it here.

\begin{conj} \label{conj:p4_optimal}
For $n \ge 6$, the maximum size of a $P_4$-intersecting family is $\frac{17}{128} \cdot 2^{\binom n2}$. Furthermore, any extremal family is isomorphic to one generated by the Christofides construction.
\end{conj}

This conjecture is particularly significant due to its connection to a more general conjecture proposed by Alon and Spencer \cite{AlonSpencer2016}. They conjectured that for any graph $H$ that is not a star-forest (a graph where each connected component is a star), the trivial $H$-intersecting family is not extremal.

\begin{conj}[Alon-Spencer \cite{AlonSpencer2016}]
There exists some constant $\epsilon > 0$ such that if $H$ is a fixed graph that is not a star-forest and $\mathcal F$ is an $H$-intersecting family on $K_t$, then
\[|\mathcal F| < \left(\frac 12 - \epsilon\right)2^{\binom t2}.\]
\end{conj}

A proof of Conjecture \ref{conj:p4_optimal} would completely resolve this conjecture. The reasoning is that any graph $H$ that is not a star-forest must contain $K_3$ or $P_4$ as a subgraph.
\begin{itemize}
    \item If $H$ contains $K_3$, any $H$-intersecting family $\mathcal{F}$ is also a $K_3$-intersecting family. Thus, $|\mathcal{F}|$ is bounded by the known maximum size for $K_3$, which is $\frac{1}{8} \cdot 2^{\binom n2}$, as proved by Ellis, Filmus, and Friedgut \cite{ellis-filmus-friedgut}.
    \item If $H$ is triangle-free but contains $P_4$, any $H$-intersecting family $\mathcal{F}$ is also a $P_4$-intersecting family. If Conjecture \ref{conj:p4_optimal} holds, $|\mathcal{F}|$ would be bounded by $\frac{17}{128} \cdot 2^{\binom n2}$.
\end{itemize}
Since both bounds (the established $\frac{1}{8} = \frac{16}{128}$ and the conjectured $\frac{17}{128}$) are strictly less than $\frac{1}{2}$, a proof of Conjecture \ref{conj:p4_optimal} would confirm the Alon-Spencer conjecture.

As partial evidence supporting Conjecture \ref{conj:p4_optimal}, we performed a computational search for denser constructions on small host graphs. This search yielded no constructions outperforming the Christofides bound. We summarize our findings as follows.

\begin{theorem}
    The Christofides density of $\frac{17}{128}$ is optimal for all $P_4$-intersecting families generated from a connected host graph $G$ on 6 vertices with 7 or 8 edges.
\end{theorem}

\begin{proof}
    Our method involved an exhaustive search using Python (our implementation is available in Appendix~\ref{appendix:code}). For a given host graph $G$, we generated its subgraphs with at least 3 edges (the number of edges in $P_4$) and constructed an auxiliary graph whose vertices represented these subgraphs. An edge was placed between two vertices in the auxiliary graph if their corresponding subgraphs had a $P_4$ in their intersection. The size of the largest $P_4$-intersecting family supported by $G$ is then the size of the maximum clique in this auxiliary graph, a problem known to be NP-complete \cite{Karp1972}. 
\end{proof}

While our search was limited to small graphs, its results support the conjecture that the Christofides density is optimal.

\section*{Acknowledgments}
The authors would like to thank Richard Chen, Alexander Wang, Yuval Wigderson, and Cameron Wright for helpful discussions. We would also like to thank William Linz (for bringing his paper to our attention) and Varun Sivashanker (for catching a minor error) for comments on an earlier draft. Finally, we thank the organizers of the Institute for Advanced Study\slash Park City Mathematics Institute 2025 Summer School in Probabilistic and Extremal Combinatorics for the invitation to participate in the program where this research was conducted.

\appendix \section{Code}\label{appendix:code}

The code can be found on Github at \href{https://github.com/paulhamr/intersecting-families}{https://github.com/paulhamr/intersecting-families}.

\end{document}